\def\DateTime{30/January/2020}
\def\Version{Version $1.0$}
\def\yes{\if00}

\def\ifquery{\yes}

\documentclass[reqno]{amsart}

\usepackage{amsmath}
\usepackage{amssymb}
\usepackage{amscd}
\usepackage{verbatim}
\usepackage{mathrsfs,enumerate}
\usepackage{color}
\usepackage[normalem]{ulem}
\usepackage{bbm}

\usepackage[sc]{mathpazo} 
\usepackage[T1]{fontenc}

\theoremstyle{plain}
\newtheorem{Theorem}{Theorem}[section]
\newtheorem{Proposition}[Theorem]{Proposition}
\newtheorem{Lemma}[Theorem]{Lemma}
\newtheorem{Corollary}[Theorem]{Corollary}

\theoremstyle{definition}
\newtheorem{Definition}[Theorem]{Definition}
\newtheorem{Conjecture}[Theorem]{Conjecture}
\newtheorem{Remark}[Theorem]{Remark}

\numberwithin{equation}{section}

\def\rom{\textup}
\newcommand{\ZZ}{{\mathbb{Z}}}
\newcommand{\QQ}{{\mathbb{Q}}}
\newcommand{\RR}{{\mathbb{R}}}

\newcommand{\CC}{{\mathbb{C}}}
\newcommand{\PP}{{\mathbb{P}}}

\newcommand{\Proj}{\operatorname{Proj}}

\newcommand{\ord}{\operatorname{ord}}

\newcommand{\ndot}{\raisebox{.4ex}{.}}

\newcommand{\fin}{\operatorname{fin}}
\newcommand{\indic}{1\hspace{-0.25em}\mathrm{l}}

\def\colorsout#1{\bgroup\markoverwith{\textcolor{#1}{\rule[0.5ex]{2pt}{0.7pt}}}\ULon} 
\ifquery
\def\query#1{\setlength\marginparwidth{65pt} 
\marginpar{\raggedright\fontsize{7.81}{9} 
\selectfont\upshape\hrule\smallskip 
#1\par\smallskip\hrule}}

\else
\def\query#1{}

\fi
\definecolor{ruby}{rgb}{0.88, 0.07, 0.37}
\definecolor{coolblack}{rgb}{0.0, 0.18, 0.39}
\definecolor{darkspringgreen}{rgb}{0.09, 0.45, 0.27}
\definecolor{emerald}{rgb}{0.31, 0.78, 0.47}
\definecolor{lavenderindigo}{rgb}{0.58, 0.34, 0.92}
\definecolor{mred}{rgb}{0.83, 0.0, 0.0}
\DeclareSymbolFont{bbold}{U}{bbold}{m}{n}
\DeclareMathSymbol{\bbalpha}{\mathord}{bbold}{"0B}
\DeclareMathSymbol{\bbbeta}{\mathord}{bbold}{"0C}
\DeclareMathSymbol{\bbgamma}{\mathord}{bbold}{"0D}
\DeclareMathSymbol{\bbdelta}{\mathord}{bbold}{"0E}
\DeclareMathSymbol{\bbespilon}{\mathord}{bbold}{"0F}
\DeclareMathSymbol{\bbzeta}{\mathord}{bbold}{"10}
\DeclareMathSymbol{\bbeta}{\mathord}{bbold}{"11}
\DeclareMathSymbol{\bbtheta}{\mathord}{bbold}{"12}
\DeclareMathSymbol{\bbiota}{\mathord}{bbold}{"13}
\DeclareMathSymbol{\bbkappa}{\mathord}{bbold}{"14}
\DeclareMathSymbol{\bblambda}{\mathord}{bbold}{"15}
\DeclareMathSymbol{\bbmu}{\mathord}{bbold}{"16}
\DeclareMathSymbol{\bbnu}{\mathord}{bbold}{"17}
\DeclareMathSymbol{\bbxi}{\mathord}{bbold}{"18}
\DeclareMathSymbol{\bbpi}{\mathord}{bbold}{"19}
\DeclareMathSymbol{\bbrho}{\mathord}{bbold}{"1A}
\DeclareMathSymbol{\bbsigma}{\mathord}{bbold}{"1B}
\DeclareMathSymbol{\bbtau}{\mathord}{bbold}{"1C}
\DeclareMathSymbol{\bbupsilon}{\mathord}{bbold}{"1D}
\DeclareMathSymbol{\bbphi}{\mathord}{bbold}{"1E}
\DeclareMathSymbol{\bbchi}{\mathord}{bbold}{"1F}
\DeclareMathSymbol{\bbpsi}{\mathord}{bbold}{"20}

\begin{document}

\title[Toward Fermat's conjecture over arithmetic function fields]%
{Toward Fermat's conjecture over \\arithmetic function fields}
\author{Atsushi Moriwaki}
\email{moriwaki@math.kyoto-u.ac.jp}
\address{Department of Mathematics, Faculty of Science, Kyoto University, Kyoto, 606-8502, Japan}
\thanks{\hskip-1em{\em Date and Version}: \DateTime\ (\Version)}

\begin{abstract}
Let $K$ be an arithmetic function field, that is, a field of finite type over $\QQ$.
In this note, as an application of the height theory due to Chen-Moriwaki \cite{CMArakelovAdelic},
we would like to show that
the solutions of Fermat's curve $X^N + y^N = 1$ of degree $N$ over $K$ consist of
only either $0$ or roots of unity for almost positive integers $N$.
More precisely, 
the density of such $N$ in $\ZZ_{\geqslant 1}$ is $1$.
\end{abstract}

\maketitle

\section*{Conclusion}

In this note, we would like to show an application of the new height theory due to Chen-Moriwaki
\cite{MoARH,CMArakelovAdelic} to Fermat's conjecture over an arithmetic function field.

Let $N$ be a positive integer and 
$F_N$ be the \emph{Fermat curve of degree $N$} over $\ZZ$, that is, 
\[
F_N = \Proj (\ZZ[X, Y, Z]/(X^N + Y^N - Z^N)).
\]
Let $K$ be a field.
We say that $F_N$ has \emph{Fermat's property over $K$} if
\[
\{ (x, y) \in K^2 \mid x^N + y^N = 1 \} \subseteq (\{ 0 \} \cup \mu(K))^2,
\]
where $\mu(K)$ is the group consisting of roots of unity in $K$, that is,
\[
\mu(K) = \{ x \in K \mid \text{$\exists\ n \in \ZZ_{\geqslant 1}$ such that $x^n = 1$} \}.
\]

\begin{Theorem}\label{thm:Fermat:property}
We assume the following:
\begin{enumerate}
\renewcommand{\labelenumi}{\textup{(\roman{enumi})}}
\item The field $K$ has a proper adelic structure  
with Northcott's property
(see Section~\ref{sec:adelic:structure}).
\item There is a positive integer $p_0$ such that $F_{p}(K)$ is finite for any prime number $p \geqslant p_0$.
\end{enumerate}
Then one has
\[
\lim_{m\to\infty} 
\frac{\#\left( \{ N \in \ZZ \mid \text{$1 \leqslant N \leqslant m$ and $F_N$ has Fermat's property over $K$} \}\right)}{m} = 1.
\]
\end{Theorem}

If $K$ is an arithmetic function field, that is, $K$ is a finitely generated field over $\QQ$, then the first condition (i) of the above theorem holds by Chen-Moriwaki \cite{MoARH,CMArakelovAdelic} or 
Proposition~\ref{prop:adelic:structure:arithmetic:function:field}.
Moreover, the second condition (ii) also holds for $p_0 = 5$ by Faltings \cite{FaGen}.
Therefore, as a consequence of the above theorem, one has the following corollary.

\begin{Corollary}
If $K$ is an arithmetic function field,  then
\[
\lim_{m\to\infty} 
\frac{\#\left( \{ N \in \ZZ \mid \text{$1 \leqslant N \leqslant m$ and $F_N$ has Fermat's property over $K$} \}\right)}{m} = 1.
\]
\end{Corollary}

In the case where $K = \QQ$, it was proved by \cite{Filaseta,Granville,HeathBrown} (cf. \cite{FermatRibenboim}).
A general number field case is treated in \cite{IKMFaltings}.
The above corollary gives an evidence of the following conjecture:

\begin{Conjecture}[Fermat's conjecture over an arithmetic function field]
Let $K$ be an arithmetic function field. Then is
there a positive integer $N_0$ depending on $K$ such that
$F_N$ has Fermat's property over $K$ for all $N \geqslant N_0$?
\end{Conjecture}

\section{Adelic structure of field}
\label{sec:adelic:structure}

Here we recall an adelic structure of a field and the height function with respect to the adelic structure
(for details, see \cite{CMArakelovAdelic}).

Let $K$ be a field.
An \emph{adelic structure} of $K$ consists of data
$S = (K, (\Omega, \mathcal A, \nu), \phi)$ satisfying the following properties:
\begin{enumerate}
\renewcommand{\labelenumi}{\textup{(\arabic{enumi})}}
\item $(\Omega, \mathcal A, \nu)$ is a measure space, that is,
$\mathcal A$ is a $\sigma$-algebra of $\Omega$ and $\nu$ is a measure on $(\Omega, \mathcal A)$.

\item 
The last $\phi$ is a map
from $\Omega$ to $M_K$, where $M_K$ is the set of all absolute values of $K$. We denote the absolute value $\phi(\omega)$ ($\omega \in \Omega$) by $|\ndot|_{\omega}$.

\item For $a \in K^{\times}$, the function $(\omega \in \Omega) \mapsto \log |a|_{\omega}$ is $\nu$-integrable.
\end{enumerate}
We call the data $S$ an \emph{adelic curve}.
Moreover, $S$ is said to be \emph{proper} if
\begin{equation}\label{eqn:product:formula}
\int_{\Omega} \log |a|_{\omega} \nu(d\omega) = 0
\end{equation}
holds for all $a \in K^{\times}$. The equation \eqref{eqn:product:formula} is
called the \emph{product formula}.

From now on, we assume that $S$ is proper.
For $(x_0, \ldots, x_n) \in K^{n+1} \setminus \{ (0,\ldots,0) \}$,
we define $h_S(x_0, \ldots, x_n)$ to be
\[
h_S(x_0, \ldots, x_n) := \int_{\Omega} \log \max \{ |x_0|_{\omega},\ldots,
|x_n|_{\omega} \} \nu(d\omega).
\]
Note that by the product formula \eqref{eqn:product:formula},
the map $h_S : K^{n+1} \setminus \{ (0,\ldots,0) \} \to \RR$
descents to $\PP^n(K) \to \RR$. By abuse of notation, we denote it by
$h_S(x_0 :  \cdots : x_n)$, which is called the \emph{height} of 
$(x_0 :  \cdots : x_n)$ with respect to $S$.
Moreover, we say $S$ has \emph{Northcott's property} if, for any $C \in \RR$,
the set 
\[
\left\{ a \in K \ \left|\ \int_{\Omega} \log \max \{  |a|_{\omega}, 1 \} v(d\omega) \leqslant C \right\}\right.
\]
is finite. Note that if $x_i \not= 0$, then \[ 
\int_{\Omega} \log \max \{  |x_j/x_i|_{\omega}, 1 \} v(d\omega) = h_S(x_i: x_j)
\leqslant h_S(x_0 : \cdots : x_n)\] for all 
$j = 0, \ldots n$, so that
if $S$ has Northcott's property,
then the set
\[
\{ (x_0 : \cdots : x_n) \in \PP^n(K) \mid h_S(x_0 : \cdots : x_n) \leqslant C \}.
\]
is finite for any $C \in \RR$. Finally note that
if $K$ is an arithmetic function field (i.e., $K$ is a finitely generated field over $\QQ$), then
$K$ has a proper adelic structure with Northcott's property (cf. \cite{MoARH,CMArakelovAdelic} or 
Proposition~\ref{prop:adelic:structure:arithmetic:function:field}
for a construction without Arakelov theory).

\section{Proof of Theorem~\ref{thm:Fermat:property}}

Let $S = (K, (\Omega, \mathcal A, \nu), \phi)$ be a proper adelic structure of a field $K$ with Northcott's property. 

\begin{Lemma}\label{lem:height:zero:root:of:unity}
If $h_S(a_0 : \cdots : a_n) = 0$ for $(a_0 : \cdots : a_n) \in \PP^n(K)$,
then there is $\lambda \in K^{\times}$ such that
$\lambda a_i \in \{ 0 \} \cup \mu(K)$ for all $i=0,\ldots,n$.
\end{Lemma}

\begin{proof}
Without of loss of generality, we may assume that $a_0 = 1$.
Note that \[h_S(a_0^N : \cdots : a_n^N) = N h_S(a_0 : \cdots : a_n) = 0,\]
so that, by Northcott's property, the set $\{ (a_1^N, \ldots, a_n^N) \mid
N \in \ZZ_{\geqslant 1} \}$ is finite. Therefore, there are $N, N' \in \ZZ_{\geqslant 1}$ such that $N < N'$ and
$(a_1^N, \ldots, a_n^N) = (a_1^{N'}, \ldots, a_n^{N'})$. Thus, if $a_i \not= 0$,
then $a_i^{N' - N} = 1$, as required.
\end{proof}

\begin{Proposition}\label{prop:Fermat:property:equiv}
The following are equivalent:
\begin{enumerate}
\renewcommand{\labelenumi}{\textup{(\arabic{enumi})}}
\item $F_N$ has Fermat's property over $K$.

\item $F_N(K) \subseteq \{ (x : y : z) \in \PP^2(K) \mid h_S(x : y : z) = 0 \}$.
\end{enumerate}
\end{Proposition}

\begin{proof}
(1) $\Longrightarrow$ (2) : Let $(x : y : z) \in F_N(K)$.
If $z \not= 0$, then
$(x/z)^N + (y/z)^N = 1$, so that $x/y, y/z \in \{ 0 \} \cup \mu(K)$, and hence
$h_S(x/z : y/z : 1) = 0$. If $z = 0$, then there is $\zeta \in \mu(K)$ such that $y = \zeta x$. Thus $h_S(x : y : 0) = h_S(1 : \zeta : 0) = 0$.

\medskip
(2) $\Longrightarrow$ (1) : Let $(x, y) \in K^2$ such that $x^N + y^N = 1$.
Then $(x : y : 1) \in F_N(K)$, so that $h_S(x : y : 1) = 0$.
Thus, by Lemma~\ref{lem:height:zero:root:of:unity},
there is $\lambda \in K^{\times}$ such that $\lambda x, \lambda y, \lambda
\in \{ 0 \} \cup \mu(K)$.
In particular, $\lambda \in \mu(K)$, so that $x, y \in \{0\} \cup \mu(K)$.
\end{proof}

\begin{Proposition}\label{prop:Fermat:prop:large}
If $F_{N}(K)$ is finite, then there is a positive integer $m_0$ such that
$F_{Nm}$ has Fermat's property over $K$ for all $m \geqslant m_0$.
\end{Proposition}

\begin{proof}
Let us consider $H$ and $a$ as follows:
\[
\begin{cases}
H = \max \{ h_S(x) \mid x \in F_N(K) \},\\
a = \inf \left\{ h_S(x) \mid \text{$x \in \PP^2(K)$ and $h_S(x) > 0$} \right\}.
\end{cases}
\]
Note that $a > 0$ because $S$ has Northcott's property.
If $m \in \ZZ_{\geqslant \lceil \exp(H/a) \rceil}$ and there is $(x : y : z) \in F_{Nm}(K)$
with $h_S(x : y : z) > 0$, then $h_S(x:y:z) \geqslant a$ and $(x^m : y^m : z^m) \in F_N(K)$, so that
\[
H \geqslant h_S(x^m : y^m : z^m) = mh_S(x:y:z) \geqslant ma,
\]
and hence $\exp(H/a) \geqslant \exp(m)$. Thus $m \geqslant \exp(m)$,
which is a contradiction because $m \geqslant 1$.
Therefore $F_{Nm}$ has Fermat's property over $K$ by Proposition~\ref{prop:Fermat:property:equiv}.
\end{proof}

\begin{Lemma}\label{lem:probability:dist}
Let $T$ be a subset of $\ZZ_{\geqslant 1}$ such that 
there exists a positive integer $p_0$ with the following property: for any prime $p \geqslant p_0$, one can find $m_p \in \ZZ_{\geqslant 1}$ depending on $p$ such that
$p\ZZ_{\geqslant m_p} \subseteq T$. 
Then \[ \lim_{m\to\infty} \frac{T \cap [1, m]}{m} = 1.\]
\end{Lemma}

\begin{proof}
Since the Riemann zeta function has a pole at $1$,
for $0 < \epsilon < 1$, there are 
primes $p_1, \ldots, p_r$ such that $p _0 \leqslant p_1 < \cdots < p_r$ and
$\prod_{i=1}^r (1 - 1/p_i) \leqslant \epsilon$. By our assumption, if we set $Q = p_1 \cdots p_r$, then there is $n_0 \in \ZZ_{\geqslant 1}$ such that \[\{ n \in \ZZ_{\geqslant 1} \mid \text{$n \geqslant n_0$ and $\mathrm{GCD}(Q, n) \not= 1$} \} \subseteq T,\]
that is,
\[
\ZZ_{\geqslant 1} \setminus T \subseteq \{ n \in \ZZ_{\geqslant 1} \mid n < n_0 \} \cup
\{ n \in \ZZ_{\geqslant 1} \mid \mathrm{GCD}(Q, n) = 1 \},
\]
so that
\[
\#\left( (\ZZ_{\geqslant 1} \setminus T) \cap [1, m] \right) \leqslant (n_0-1) + \#\{ 1 \leqslant n \leqslant m \mid \mathrm{GCD}(Q, n) = 1 \}.
\]
As $\#\{ 1 \leqslant n \leqslant Q \mid \mathrm{GCD}(Q, n) = 1 \}$ is equal to the Euler number $\varphi(Q)$,
we obtain
\begin{multline*}
\#\{ 1 \leqslant n \leqslant m \mid \mathrm{GCD}(Q, n) = 1 \}  \\
\leqslant (m/Q + 1) \varphi(Q) = (m + Q) (1-1/p_1) \cdots (1 - 1/p_r) \leqslant (m+Q)\epsilon.
\end{multline*}
Thus, for $m \geqslant \max \{ (n_0-1)/\epsilon, Q/\epsilon \}$,
\begin{align*}
1 \geqslant \frac{\#(T \cap [1, m])}{m} & = 1 - \frac{\#((\ZZ_{\geqslant 1} \setminus T) \cap [1, m])}{m}
\geqslant 1 - \frac{(n_0-1) + (m+Q)\epsilon}{m} \\
& = 1 -\left( \frac{n_0-1}{m} + \left(1 + \frac{Q}{m}\right) \epsilon \right) 
\geqslant 1 - (\epsilon + (1 + \epsilon) \epsilon) \geqslant 1 - 3\epsilon,
\end{align*}
as required.
\end{proof}

\begin{proof}[Proof of Theorem~\ref{thm:Fermat:property}]
Theorem~\ref{thm:Fermat:property} follows from Proposition~\ref{prop:Fermat:prop:large} and
Lemma~\ref{lem:probability:dist}.
\end{proof}

\begin{Remark}
(1) If $x^N + y^N = 1$ and $x, y \in \mu(K)$, then it is easy to see that
$(x^N, y^N) = (e^{\pi i/3}, e^{-\pi i/3})$ or $(e^{-\pi i/3}, e^{\pi i/3})$.

\medskip
(2) Similarly as the original Fermat's conjecture, we may consider the following problem:
for an arithmetic function field $K$ and a sufficiently large integer $N$,
does $x^N + y^N = 1$ ($x, y \in K$) imply $xy = 0$?
However, it doesn't seem to be good as a conjecture.
Indeed, we assume that $e^{\pi i/3}, e^{-\pi i/3} \in K$. 
If we set $x = e^{\pi i/3}$ and $y = e^{-\pi i/3}$, then
$x + y = 1$, $x^6 = y^6 = 1$, $x^5 = y$ and $y^5 = x$, so that
$x^N + y^N = 1$ and $xy \not= 0$ for an integer $N$ with $N \equiv 1, 5 \ \mathrm{mod}\ 6$.
\end{Remark}

\appendix
\section{Adelic structure of arithmetic function field}\label{sec:adelic:structure:arithmetic:function:field}
In this appendix we give a specific adelic structure of an arithmetic function field without using Arakelov Geometry. 

\medskip
Let $\mathcal{A}_{[0,1]^n}$ and $\nu_{[0,1]^n}$ be the Borel $\sigma$-algebra and
the standard Borel measure on $[0,1]^n$, respectively.
Let $p_n : [0,1]^n \to \CC^n$ be a continuous map given by $p_n(t_1, \ldots, t_n) = (e^{2\pi i t_1}, \ldots, e^{2\pi i t_n})$. 
Let $f \in \CC[X_1, \ldots, X_n]$. Note that if $f \not= 0$, then
$\log |f(e^{2\pi i t_1}, \ldots, e^{2\pi i t_n})|$ is integrable on $[0,1]^n$, so that
one can introduce $\bbmu(f )$ as follows:
\[
\bbmu(f ) := \begin{cases}
-\infty & \text{if $f = 0$},\\[1ex]
{\displaystyle  \int_{[0,1]^n} \log |f(e^{2\pi i t_1}, \ldots, e^{2\pi i t_n}) | dt_1 \cdots dt_n} & \text{otherwise}.
\end{cases}
\]
Then one has the following basic facts, which can be checked easily.

\begin{enumerate}
\renewcommand{\labelenumi}{\textup{(\alph{enumi})}}
\item 
Let $\Omega_{\infty}$ be the set of all $(t_1, \ldots, t_n) \in [0,1]^n$ such that
$e^{2\pi i t_1}, \ldots, e^{2\pi i t_n}$ are algebraically independent over $\QQ$.
Then $\Omega_{\infty}$ is a Borel subset of $[0,1]^n$ such that
$[0,1]^n \setminus \Omega_{\infty}$ is a null set.
In particular, $\Omega_{\infty}$ is dense in $[0,1]^n$.

\item
For any $d$ and $C$, the set \[\{ f \in \ZZ[X_1, \ldots, Z_n] \mid \text{$\bbmu(f) \leqslant C$ and $\deg(f) \leqslant d$} \}\]
is finite.

\item
If $f \in \CC[X_1, \ldots, X_n] \setminus \{ 0 \}$, then
\[
\bbmu(f) \geqslant \log \min\{ |a| \mid \text{$a$ is a non-zero coefficient of $f$} \}
\]
In particular, if $f \in \ZZ[X_1, \ldots, X_n] \setminus \{ 0 \}$,
then $\bbmu(f) \geqslant 0$.
\end{enumerate}

\bigskip
$\bullet$ Let $\Omega_h$ be the set of all prime divisors on $\PP^n_{\QQ} = \Proj(\QQ[T_0, \ldots, T_n])$. If we set $X_i := T_i/T_0$ for $i=1, \ldots, n$,
then the function field of $\PP^n_{\QQ}$ is $\QQ(X_1, \ldots, X_n)$.
For each $\omega \in \Omega_h$, let $P_{\omega}$ be a defining homogeneous
polynomial of $\omega$ such that $P_{\omega} \in \ZZ[T_0, \ldots, T_n]$ and $P_{\omega}$ is primitive. Note that $P_{\omega}$ is uniquely determined up to $\pm 1$.
We fix $\lambda \in \RR_{\geqslant 0}$. For $\omega \in \Omega_h$, a nonarchimedean absolute value $|\ndot|_{\omega}$ on $\QQ(X_1, \ldots, X_n)$ is defined to be
\[
| f|_{\omega} := \exp\left( \lambda \deg(P_{\omega}) +  \bbmu(P_\omega(1, X_1, \ldots, X_n))\right)^{-\ord_{\omega}(f)}\quad (\forall f \in \QQ(X_1, \ldots, X_n))
\]
Note that if $\lambda = 0$ and $\bbmu(P_\omega(1, X_1, \ldots, X_n)) = 0$,
then $|\ndot|_{\omega}$ is the trivial absolute value.

Let $\Omega_v$ be the set of all prime number of $\ZZ$. For $p \in \Omega_v$, 
let $|\ndot|_p$ be the $p$-adic absolute value of $\ZZ$ with $|p|_p = 1/p$.
For $f = \sum_{(i_1, \ldots, i_N) \in \ZZ^n_{\geqslant 0}} a_{i_1,\ldots, i_n} X_1^{i_1} \cdots X_n^{i_n} \in \QQ[X_1, \ldots, X_n]$, if we set $|f|_p = \max_{(i_1, \ldots, i_n) \in \ZZ_{\geqslant 0}^n} \{ |a_{i_1,\ldots, i_n}|_p \}$, then,
by Gauss' lemma, one can see that $|fg|_p = |f|_p |g|_p$, so that $|\ndot|_p$ extends to an absolute value of $\QQ(X_1, \ldots, X_n)$.

We set $\Omega_{\fin} := \Omega_h \coprod \Omega_v$. 
A measure space $(\Omega_{\fin}, \mathcal{A}_{\fin}, \nu_{\fin})$ is  the discrete measure space on $\Omega_{\fin}$ such that
$\nu_{\fin}(\{ \omega \}) = 1$ for all $\omega \in \Omega_{\fin}$.

\medskip
$\bullet$ By the above (a), $\Omega_{\infty}$ is a dense Borel set in $[0,1]^n$
with $\nu_{[0,1]^n}(\Omega_{\infty}) = 1$. For $t = (t_1, \ldots, t_n) \in \Omega_{\infty}$,
$| f |_t$ is defined by \[|f|_t := |f(e^{2\pi i t_1}, \ldots, e^{2\pi i t_n})|,\] where $|\ndot|$ is the usual absolute value of $\CC$. 
A measure space $(\Omega_{\infty}, \mathcal{A}_{\infty}, \nu_{\infty})$ is defined to be the restriction of the Borel $\sigma$-algebra $\mathcal{A}_{[0,1]^n}$ and the Borel measure $\nu_{[0,1]^n}$ on $[0,1]^n$.

\begin{Definition}\label{def:adelic:structure:Q:X1:Xn}
An adelic structure of $\QQ(X_1, \ldots, X_n)$ given by 
\[\left(\Omega_{\fin}, \mathcal{A}_{\fin}, \nu_{\fin} \right) \coprod 
\left(\Omega_{\infty}, \mathcal{A}_{\infty}, \nu_{\infty}\right)
\]
is denoted by $\Sigma_{n,\lambda}$.
\end{Definition}

\begin{Lemma}\label{lem:adelic:structure:Q:X1:Xn}
\begin{enumerate}
\renewcommand{\labelenumi}{\textup{(\arabic{enumi})}}
\item $\Sigma_{n,\lambda}$ is proper, that is,
\[
\sum_{\omega \in \Omega_{\fin}} \log |f|_{\omega} + \int_{\Omega_{\infty}} \log |f|_{t} d t = 0
\]
for all $f \in \QQ(X_1, \ldots, X_n)^{\times}$.

\item If $\lambda > 0$, then $\Sigma_{n,\lambda}$ has Northcott's property.
\end{enumerate}
\end{Lemma}

\begin{proof}
(1) We set $p_{\omega} := P_{\omega}(1, X_1, \ldots, X_n) \in \ZZ[X_1, \ldots, X_n]$.
Note that if we denote $\{ T_0 = 0 \}$ by $\omega_0$, then
$\bbmu(p_{\omega_0}) = 0$ and
$\ord_{\omega_0}(\ndot)$ on $\QQ(X_1, \ldots, X_n)$
is nothing more than $-\deg(\ndot)$.
Moreover, $\deg(P_{\omega}) = \deg(p_{\omega})$ for all $\omega \in \Omega_h \setminus \{\omega_0\}$.
Since $\ZZ[X_1, \ldots, X_n]$ is a UFD, one can set
\[f = a \cdot \prod_{\omega \in \Omega_h \setminus \{ \omega_{0} \}} p_{\omega}^{\ord_{\omega}(f)}\] for some $a \in \QQ$, so that 
\begin{multline*}
\sum_{\omega \in \Omega_{\fin}} \log |f|_{\omega} = 
\sum_{\omega \in \Omega_h \setminus \{ \omega_{0} \}} -\ord_{\omega}(f) (\lambda \deg(p_{\omega}) + \bbmu( p_{\omega})) \\
\kern20em + \lambda \deg(f) + \sum_{p \in \Omega_v} - \ord_p(a) \log p \\
= \sum_{\omega \in \Omega_h \setminus \{ \omega_{0} \}} -\ord_{\omega}(f) \bbmu(p_{\omega})  + \sum_{p \in \Omega_v} - \ord_p(a) \log p \\
= -\int_{[0,1]^n} \log |f(e^{2\pi i t_1}, \ldots, e^{2\pi i t_n})| dt_1 \cdots dt_n.
\end{multline*}
On the other hand, as $[0,1]^n \setminus \Omega_{\infty}$ is a null Borel subset by the above (a),
\begin{multline*}
\int_{\Omega_{\infty}} \log |f|_{t} d t = \int_{\Omega_{\infty}} \log |f(e^{2\pi i t_1}, \ldots, e^{2\pi i t_n})| d t_1 \cdots dt_n
\\ = \int_{[0,1]^n} \log |f(e^{2\pi i t_1}, \ldots, e^{2\pi i t_n})| d t_1 \cdots dt_n,
\end{multline*}
as required.

\medskip
(2) We need to see that the set $\{ f \in \QQ(X_1, \ldots, X_n) \mid h_S(1 : f) \leqslant C \}$ is finite for any $C$.
We choose $f_1, f_2 \in \ZZ[X_1, \ldots, X_n]$ such that $f = f_2/f_1$ and $f_1$ and $f_2$ have no common factors in $\ZZ[X_1, \ldots, X_n]$.
Then, by (1),
\begin{multline*}
h_{\Sigma_{n,\lambda}}(1 : f) = h_{\Sigma_{n,\lambda}}(f_1 : f_2) = \lambda \max \{ \deg(f_1), \deg(f_2) \} \\ +
\int_{[0,1]^n} \log \max \{ |f_1(e^{2\pi i t_1}, \ldots, e^{2\pi i t_n})|, |f_2(e^{2\pi i t_1}, \ldots, e^{2\pi i t_n})|\} dt_1 \cdots dt_n,
\end{multline*}
so that if $h_{\Sigma_{n,\lambda}}(1 : f) \leqslant C$, then
\[
\max \{ \bbmu(f_1), \bbmu(f_2)  \} + \lambda \max \{ \deg(f_1), \deg(f_2) \} \leqslant C.
\]
Note that $\max \{ \bbmu(f_1) , \bbmu(f_2)  \} \geqslant 0$ by the above (c),
and hence
\[
\max \{ \deg(f_1), \deg(f_2) \} \leqslant C/\lambda
\quad\text{and}\quad
\max \{ \bbmu(f_1) , \bbmu(f_2)  \} \leqslant C.
\]
Thus (2) follows from the above (b).
\end{proof}

\begin{Proposition}\label{prop:adelic:structure:arithmetic:function:field}
For an arithmetic function field $K$, one can give a proper adelic structure $S$ of $K$ such that $S$ has Northcott's property.
\end{Proposition}

\begin{proof}
Let $n$ be the transcendental degree of $K$ over $\QQ$.
Then there are $X_1, \ldots, X_n \in K$ such that
$X_1, \ldots, X_n$ are algebraically independent over $\QQ$ and
$K$ is a finite extension over $K_0 := \QQ(X_1, \ldots, X_n)$.
Fix $\lambda > 0$ and consider the adelic structure $\Sigma_{n,\lambda}$ of
$K_0$ as in Definition~\ref{def:adelic:structure:Q:X1:Xn}.
By Lemma~\ref{lem:adelic:structure:Q:X1:Xn}, $\Sigma_{n,\lambda}$ is a proper adelic curve
with Northcott's property. 
Let $\Sigma_{n,\lambda} = (K_0, (\Omega_0, \mathcal A_0, \nu_0), \phi_0)$, $\Omega := \Omega_0 \times_{M_{K_0}} M_{K}$ and $\phi : \Omega \to M_K$ the natural map.
Moreover, let $\mathcal A$ be the smallest
$\sigma$-algebra such that the projection $\Omega \to \Omega_0$ (which is denoted by $\pi_{K/K_0}$) is measurable and
$(\omega \in \Omega) \mapsto | a |_{\omega}$ is measurable for all $a \in K$.
For a function $g$ on $\Omega$ and $\omega_0 \in \Omega_0$, 
$I_{K/K_0}(g)(\omega_0)$ is defined to be
\[
I_{K/K_0}(g)(\omega_0) := \sum_{\omega \in \pi_{K/K_0}^{-1}(\omega_0)} \frac{[K_{\omega} : K_{0, \omega_0}]}{[K : K_0]} g(\omega),
\]
where $K_{\omega}$ and $K_{0, \omega_0}$ are the completions of $K$ and
$K_{0}$ with respect to $\omega$ and $\omega_0$, respectively.
Note that $I_{K/K_0}(g)$ is a function on $\Omega_0$.
For $A \in \mathcal A$, by \cite[Theorem~3.3.4]{CMArakelovAdelic},
$I_{K/K_0}(\indic_A)$ is $\mathcal A_0$-measurable, so that
we define $\nu(A)$ to be
\[
\nu(A) := \int_{\Omega_0} I_{K/K_0}(\indic_A) \nu_0(d\omega_0).
\]
By \cite[Theorem~3.3.7 and Corollary~3.5.7]{CMArakelovAdelic},
$S = (K, (\Omega, \mathcal A, \nu), \phi)$ is a proper adelic curve
with Northcott's property.
\end{proof}

\end{document}